\documentclass[12pt]{amsart}

\usepackage{DKstyle}

\newcommand{\vertiii}[1]{{\left\vert\kern-0.25ex\left\vert\kern-0.25ex\left\vert #1
    \right\vert\kern-0.25ex\right\vert\kern-0.25ex\right\vert}}
\theoremstyle{plain}

\usepackage{caption}
\usepackage[labelfont=rm]{subcaption}

\usepackage[pagebackref=true, colorlinks]{hyperref}

\hypersetup{pdffitwindow=true,linkcolor=blue,citecolor=blue,urlcolor=blue,menucolor=blue}

\usepackage{comment}

\title[ Quantitative Oppenheim  for generic forms]{A quantitative Oppenheim Theorem for generic ternary quadratic forms}
\author{Anish Ghosh }
\author{Dubi Kelmer}

\thanks{Anish Ghosh is partially supported by ISF-UGC. Dubi Kelmer is partially supported by NSF grant DMS-1401747.}
\email{ghosh@math.tifr.res.in}
\address{School of Mathematics, Tata Institute of Fundamental Research, Homi Bhabha Road, Colaba, Mumbai 400005, India}
\email{kelmer@bc.edu}
\address{Boston College, Boston, MA}


\date{\today}%
\dedicatory{}%
\commby{}%
 
\begin{document}

\begin{abstract}
We prove a quantitative version of Oppenheim's conjecture for generic ternary indefinite quadratic forms. Our results are inspired by and analogous to recent results for diagonal quadratic forms due to Bourgain \cite{Bourgain16}.
\end{abstract}

 \maketitle
 
 \section{Introduction}
 
Let $Q$ be an indefinite quadratic form in $d\geq 3$ variables that is not a multiple of a quadratic form with rational coefficients. The Oppenheim conjecture, proved in a breakthrough paper by Margulis \cite{Margulis1989}, states that $\{Q(n)~:~n\in \Z^d\}$ is dense in $\R$.  Margulis' approach used the dynamics of unipotent flows on homogeneous spaces and was not quantitative or effective, i.e. it did not give any bounds on the size of the integer vector $n$ needed to approximate a real number up to a given accuracy. One of the main difficulties of giving an effective proof is distinguishing between rational forms (for which the statement is false) and irrational forms that are very well approximated by rational forms. Recently, Bourgain investigated a quantitative version of the Oppenheim conjecture for generic diagonal forms. We follow his notation here. Margulis' theorem implies that 
% \begin{equation}\label{e:1}
% \min_{\|n\|\leq T\}|Q(n)|\leq \delta(T)$,
% \end{equation}
% for all sufficiently large $T$. 
 there are sequences $N(k)\to\infty$ and $\delta(k)\to 0$ (depending on $Q$) such that for all sufficiently large $k$,
 \begin{equation}\label{e:effective}
 \sup_{|\xi|\leq N(k)}\min_{|n|\leq k}|Q(n)-\xi|\leq \delta(k).\end{equation}
 Earlier, in \cite{LindenstraussMargulis14},  Lindenstrauss and Margulis proved an effective version of Oppenheim's conjecture valid for \emph{all} irrational indefinite quadratic forms in three variables, with $N(k)$ and $\delta(k)^{-1}$ growing logarithmically in $k$.  One would of course expect significantly stronger results for generic forms. In \cite{Bourgain16}, Bourgain considered diagonal forms $Q_{\alpha,\beta}(x,y,z)=x^2+\alpha y^2-\beta z^2$ with $\alpha,\beta>0$. For these forms he showed that  \eqref{e:effective} holds for almost all $\alpha$, as long as $\frac{N(k)^{3}}{k^\eta \delta(k)^{11/2}}\to 0$ with $\eta<1$. Moreover, assuming the Lindel\"off hypothesis for the Riemann zeta function it holds as long as 
 $\frac{N(k)}{k^{\eta}\delta(k)^2}\to 0$.

 In this note we consider the space of all indefinite ternary quadratic forms, which comes equipped with a natural probability measure, and prove a quantitative Oppenheim theorem for generic indefinite ternary quadratic forms, similar in spirit to Bourgain's result.

 \begin{thm}\label{t:main}
Assume that there is $\eta<1$ such that $\frac{N(k)}{k^{\eta}\delta(k)^2}\to 0$ and $\frac{N(k)^{3/2}}{k^\eta\delta(k)}\to 0$ as $k\to \infty$. Then, for almost every indefinite ternary quadratic form $Q$, there is a constant $c=c(Q)>0$ and $T_0=T_0(Q)>0$ such that for all $k\geq T_0$,
$$\max_{|\xi|\leq N(k)}\min_{\|n\| \leq ck} |Q(n)-\xi|\leq \delta(k).$$
\end{thm}

\begin{rem}
The second assumption, $\frac{N(k)^{3/2}}{k^\eta\delta(k)}\to 0$, is probably not needed. We note that it is automatically satisfied as long as $N(k)\ll \delta(k)^{-2}$. However, using our method it is harder to get very good approximation of large points (i.e., when  $|\xi|\gg \delta^{-2}$) leading to the additional assumption.
\end{rem}
 
Our methods involve an effective mean ergodic theorem for semisimple groups and some results from the geometry of numbers. They are very different from those in \cite{Bourgain16}. In an earlier work, we studied the \emph{shrinking target problem} for actions of semisimple groups on homogeneous spaces and as a consequence, obtained a related optimal effective result for generic ternary forms, see Theorem $5$ in \cite{GhoshKelmer15}. More results along these lines  are proved in \cite{GhoshGorodnikNevo16, AthreyaMargulis16}. Several important papers have previously studied the study of the distribution of values of a quadratic form \cite{DaniMargulis93, EskinMargulisMozes98, Sarnak97} and sometimes the work ``quantitative" is used in the literature in this context.

 \section{Setup}
\subsection{Space of forms}
In order to make the notion of a generic ternary quadratic form explicit we use the following parametrization for the space of forms. Recall that the action of $G=\SL_3(\R)$ on ternary quadratic forms forms is given by 
$$Q^g(v)=Q(vg),$$
with $g\in \SL_3(\R)$ acting on $v\in \R^3$ linearly. We say that two forms are equivalent if $Q_1=\lambda Q_2^\g$ with $\lambda\in \R$ and  $\g\in\G=\SL_3(\Z)$. Note that equivalent forms take the same values on $\Z^3$ after scaling by a constant. To avoid the scaling ambiguity we restrict to forms of determinant one. To parametrize the space of determinant one forms up to equivalence, fix a form $Q_0(v)$ given by
\begin{equation} \label{e:Q0}Q_0(x,y,z)=x^2+y^2-z^2,
\end{equation}
and note that any determinant one indefinite ternary quadratic form is given by $Q=Q_0^g$ for some $g\in \SL_3(\R)$. Moreover, two such forms  $Q_0^{g},\;Q_0^{g'}$ are equivalent if and only if $g'=\g g$ with $\g\in \G$.  We can thus parametrize the space of determinant one indefinite ternary quadratic forms up to equivalence, by the space 
$$X_3=\SL_3(\Z)\bk \SL_3(\R).$$
We recall that $X_3$ also parametrizes the space of unimodular lattices in $\R^3$, explicitly, a point $x=\G g\in X_3$ corresponds to the lattice $\Lambda=\Z^3g$ and to the form $Q=Q_0^g$. The probability measure $\mu$ on $X_3$ coming from Haar measure on $\SL_3(\R)$ gives us a natural measure on the space of forms.\\

\subsection{The $\SL_2(\R)$ action}
Let $H=\SL_2(\R),\; G=\SL_3(\R)$ and $\G=\SL_3(\Z)$. Let $Q_0$ be as in \eqref{e:Q0} and consider the  double spin cover map $\iota:H\to \SO_{Q_0}$ given by
\begin{equation}\label{e:spin}
\iota\begin{pmatrix} a& b\\ c & d\end{pmatrix}=
\begin{pmatrix}\frac{a^2-b^2-c^2+d^2}{2} & ac-bd & \frac{a^2-b^2+c^2-d^2}{2}\\
ab-cs &bc+ad& ab+cd\\
\frac{a^2+b^2-c^2-d^2}{2}& ac+bd& \frac{a^2+b^2+c^2+d^2}{2}
\end{pmatrix}
\end{equation}
This gives an irreducible right action of $H$ on $G$ and hence also on $X_3=\G\bk G$.\\

\noindent We give a norm on $H$ using the spin cover map by defining
\begin{equation}\label{e:norm}
\norm{h}:=\|\iota (h)^{-1}\|_2
\end{equation}
where $\|\cdot\|_2$ denotes the the Hilbert-Schmidt norm on $G$ given by 
$\norm{g}_2^2=\tr(g^tg)$. 
We denote $F(t)\asymp G(t)$ if there is some constant $c>1$ such that 
 $$c^{-1}F(t)\leq G(t)\leq cF(t).$$
\begin{lem}
With this norm we have that 
$$m(H_t)\asymp t$$
where $m$ denotes Haar measure on $H$.
\end{lem} 
\begin{proof}
Consider the $KA^+K$ decomposition of $H$ with $K=\SO(2)$ and 
$$A^+=\left\{a_t=\begin{pmatrix} e^{t/2} & 0\\ 0 & e^{-t/2}\end{pmatrix}: t\geq 0\right\}.$$
Let $k_\theta=\left(\begin{smallmatrix} \cos(\theta) & \sin(\theta)\\ -\sin(\theta)& \cos(\theta)\end{smallmatrix}\right)$ parametrize $K$.
In the coordinates $h=k_\theta a_tk_{\theta'}$  the Haar measure of $H$ is given by
$$dm(h)=\sinh(t)d\theta d\theta'dt,$$
Moreover, a direct computation using \eqref{e:spin} shows that for $h=k_\theta a_tk_{\theta'}$ we have
$$\norm{h}^2=1+2(1+2\sinh^2(t))(1-\sin^2(\theta)\cos^2(\theta)).$$
In particular, for large $t\gg1$ we have that $\|h\|\asymp e^t$ and hence $m(H_t)\asymp t$ as claimed.
\end{proof}

 \subsection{Mean Ergodic Theorem} \label{s:MET}
 We note that the representation $\iota:\SL_2(\R)\to \SL_3(\R)$ defined in \eqref{e:spin} is irreducible and hence the $H$ action on $X_3$ 
 satisfies an effective mean ergodic theorem (see \cite{GhoshGorodnikNevo14, GorodnikNevo10}) with best possible rate. Let 
 $H_t=\{h\in H: \|h\|\leq t\}$ and 
 for $f \in L^{2}(X_3)$, consider the unitary averaging operator:
 \begin{equation}
 \pi_{t}(f)(x) := \frac{1}{m(H_t)}\int_{H_t}f(x\iota(h))dm(h).
 \end{equation}
We then have 
\begin{thm}
For any  $0<\kappa<1/2$, for any $f \in L^{2}(X_3)$, 
 \begin{equation}\label{eq:met}
 \|\pi_{t}f - \int_{X_3}fd\mu\|_{2} \ll_\kappa m(H_t)^{-\kappa}\|f\|_2.
 \end{equation}
\end{thm}
As a direct consequence, we can estimate the measure of point $x\in X_3$ whose orbit $xH_t=\{x\iota(h): \|h\|\leq t\}$ miss a small set $B\subseteq X_3$. Explicitly, we have
\begin{cor}\label{c:Shrinking}
For any $\eta<1$, for any $B\subseteq X_3$
$$\mu\{x\in X_3: xH_t\cap B=\emptyset\}\ll_\eta \frac{1}{t^{\eta}\mu(B)}.$$
\end{cor}
\begin{proof}
Let $\cC_{t,B}=\{x\in X_3: xH_t\cap B=\emptyset\}$.
Let $f$ denote the indicator function of $B$, and note that if $x\in \cC_{t,B}$ then $\pi_t(f)(x)=0$.
Applying the mean ergodic theorem with $\eta=2\kappa$ we get
\begin{eqnarray*}
\mu(B)^2\mu(\cC_{t,B})&=&\int_{\cC_{t,B}}|\pi_t(f)(x)-\mu(B)|^2d\mu(x)\\
&\leq &\|\pi_t(f)(x)-\mu(B)\|^2\ll_\eta  \frac{\|f\|_2^2}{m(H_t)^{\eta}}\ll \frac{\mu(B)}{t^\eta}.
\end{eqnarray*}
and dividing both sides by $\mu(B)^2$ gives the desired estimate. 
\end{proof}

 \subsection{Unfolding}
 We need the following result relating small sets in $\R^3$ to corresponding sets on $X_3$. Explicitly, thinking of $X_3$ as the space of unimodular lattices, for any set $\Omega\subseteq \R^3$ let $B_\Omega\subset X_3$ be defined as 
 $$B_\Omega=\{\Lambda\in X_3: \Lambda \cap \Omega\neq \emptyset\}.$$
 We then have the following result 
 \begin{lem}\label{l:VolBound}
For any set $\Omega\subseteq \R^3$ we have
 $$\mu(B_\Omega)\geq \min \left\{\frac{1}{5},\frac{\vol(\Omega)}{5}\right\}$$% \left\lbrace\begin{array}{cc} \vol(\Omega)/ 5& \vol(\Omega)<1\\ 1/5& \vol(\Omega)\geq 1 \end{array}\right.$$
 \end{lem}
 \begin{proof}
Let $f$ denote the indicator function of $\Omega$ and let $\hat{f}$ be its Siegel transform,
$$\hat{f}(\Lambda)=\sum_{0\neq v\in \Lambda} f(v).$$ By \cite[Lemma 1]{Rogers1956} we have the identity
 $$\int_{X_3} |\hat{f}(\Lambda)|^2d\mu=\left| \int_{\R^3} f(x)dx\right|^2+\sum_{(p,q)=1}\int_{\R^3}f(px)f(qx)dx.$$
 We can bound the second term (crudely) by 
 \begin{eqnarray*}
 \sum_{p=1}^\infty\sum_{q=1}^\infty \int_{q^{-1}\Omega}f(px)dx&\leq & \sum_{p=1}^\infty\frac{1}{p^3}\sum_{q=1}^\infty \vol(pq^{-1}\Omega\cap \Omega)\\
 &\leq &\sum_{p=1}^\infty\frac{1}{p^3}\left(\sum_{q=1}^p\vol(\Omega)+\sum_{q=p}^\infty\vol(\frac{p}{q}\Omega)\right)\\
 &\leq &2\vol(\Omega)\sum_{p=1}^\infty \frac{1}{p^2}\leq 4\vol(\Omega),
 \end{eqnarray*}
 and hence 
 $$\int_{X_3} |F(\Lambda)|^2d\mu\leq \vol(\Omega)(\vol(\Omega)+4).$$
 Now by Siegel's mean value theorem 
 $$\int_{X_3}F(\Lambda)d\mu=\int_{\R^3}f(x)dx=\vol(\Omega)$$ 
 and since $F(\Lambda)=0$ for $\Lambda\not\in B_\Omega$ we get
 $$\vol(\Omega)^2= \left|\int_{B_\Omega}F(\Lambda)d\mu\right|^2\leq \mu(B_\Omega)\int_{X_3}|F(\Lambda)|^2d\mu\leq  \mu(B_\Omega)\vol(\Omega)(\vol(\Omega)+4).$$
From this we get the desired lower bound
 $$\mu(B_\Omega)\geq \frac{\vol(\Omega)}{\vol(\Omega)+4}\geq  \min \left\{\frac{1}{5},\frac{\vol(\Omega)}{5}\right\}.$$
 \end{proof}

\subsection{Shrinking targets}
Next we define the shrinking targets. For parameters $\delta\in (0,1)$ and $\xi\in \R,$ let $L=\max\{1,|\xi|^{1/2}\}$ and define
$$\Omega_{\xi,\delta}=\{(x,y,z)\in \R^3: |x|\leq L, |y|\leq L,\; |\sqrt{|x^2+y^2-\xi|}-z|\leq 
\tfrac{\delta}{2L}\}.$$
We note that $\Omega_{\xi,\delta}$ has volume $4\delta L$ and that  $|Q_0(v)-\xi|\leq \delta$ for any $v\in \Omega_{\xi,\delta}$.
We now define our shrinking targets as 
$$B_{\xi,\delta}=\{\Lambda \in X_3: \Lambda \cap \Omega_{\xi,\delta}\neq \emptyset\}.$$
Then by Lemma \ref{l:VolBound} we have that $\mu(B_{\xi,\delta,})\gg \min\{1,\delta L\}$.

\section{Proofs}
The proof of Theorem \ref{t:main} follows from the following effective result, bounding the measure of the set of forms that do not well approximate a point.
\begin{prop} \label{p:FiniteBound}
There is a set $\cC_{k,\delta,\xi}\subseteq X_3$ of measure at most $O(\max\{\frac{|\xi|^{1/2}}{k}, \frac{1}{k\delta}\})$ such that for any $\G g\in X_3\setminus \cC_{k,\delta,\xi}$ we have
$$\min_{\|n\|\leq \|g\|k}|Q_0^g(n)-\xi|\leq \delta.$$
\end{prop}
\begin{proof}
Let $\eta<1$, then by Corollary \ref{c:Shrinking} we have that
$\mu(\{x\in X_3: xH_k\cap B_{\delta,\xi}=\emptyset\})\ll \frac{1}{k^{\eta}\mu(B_{\delta,\xi})})$.
Now let $x=\G g$ be such that $xH_k\cap B_{\xi,\delta}\neq \emptyset$ then there is some $h\in H_k$ and $n\in \Z^3$ such that 
$ngh\in \Omega_{\xi,\delta}$ so that $|Q_0(ng)-\xi|\leq \delta$ and also $\|ngh\|\leq L$ with $L=\max\{|\xi|^{1/2},1\}$. We can then bound
$$\|n\|=\|nghh^{-1}g^{-1}\|\leq \|ngh\|\|h\|\|g\|\leq \|g\|Lk.$$

Let $\tilde{k}=L k$ then $\|n\|\leq \|g\|\tilde{k}$, $|Q_0^g(n)-\xi|\leq \delta$ and this holds for all but $O(\frac{L}{\tilde{k}^\eta \mu(B_{\delta,\xi})})=O(\max\{\frac{L}{\tilde k^\eta},\frac{1}{\tilde k^{\eta}\delta}\})$ of the points $\G g\in X_3$. Finally, note that if $\frac{L}{\tilde k^\eta}\geq \frac{1}{\tilde k^{\eta}\delta}$ then $L>1$ which happens only when  $L=|\xi|^{1/2}>1$.
\end{proof}

\begin{proof}[Proof of Theorem \ref{t:main}]
For each $k\in \N$ let $-N(k)<\xi_{k,1}<\ldots<\xi_{k,M(k)}<N(k)$ be $\delta(k)/2$ dense (so $M(k)\asymp N(k)/\delta(k)$. Fix a fundamental domain $\cF\subset \SL_3(\R)$ for $X_3$ and let $\cC\subseteq 
\cF$ denote the set of all points $g\in \cF$ such that for any $T>0$ there is $k>T$ and $\xi_{k,i}$ such that  $\min_{|n|\leq k}|Q_0^g(n)-\xi_{i,k}|>\frac{\delta(k)}{2}$, that is
$$\cC=\bigcap_{T>0}\bigcup_{k\geq T}\bigcup_{i=1}^{M(k)}\cC_{k,i},$$
and $\cC_{k,i}$ is the set of points such that $\min_{|n|\leq k}|Q_0^g(n)-\xi_{k,i}|>\frac{\delta(k)}{2}$.

Now splitting into dyadic intervals, note that for any $L=2^n$ the set 
$$\tilde C_L=\bigcup_{k=L}^{2L}\bigcup_{i=1}^{M(2L)}\cC_{k,i}\subseteq\bigcup_{i=1}^{M(2L)}\{\G g: \min_{|n|\leq L}|Q_0^g(n)-\xi_{2L,i}|>\delta(2L)\},$$
and hence, by Proposition \ref{p:FiniteBound} with $\eta=1-\epsilon/2$, we have that 
$$\mu(\tilde\cC_L)\ll\max\{\frac{|N(2L)|^{3/2}}{L^\eta \delta(2L)}, \frac{N(2L)}{L^\eta \delta(2L)^2}\})\ll L^{-\epsilon/2},$$ 
Since we can write
$\cC=\bigcap_{T>0}\bigcup_{l\geq \log(T)}\tilde \cC_{2^l}$ we can bound
$$\mu(C)\leq \sum_{l\geq \log(T)}\mu(\tilde \cC_{2^l})\ll  \sum_{l\geq \log(T)}2^{-l \epsilon/2}\ll T^{-\epsilon/2}.$$
This holds for all $T$, hence $\mu(\cC)=0$.

Now for any $g\in \cF\setminus \cC$ there is $T_0=T_0(g)$ such that for any $k\geq T_0$, for any $\xi\in \R$ with $\xi\leq N(k)$ there is $\xi_{k,i}\in (\xi-\frac{\delta(k)}{2},\xi+\frac{\delta(k)}{2})$ and $n\in \Z^3$ with $\|n\|\leq \|g\| k$ such that $|Q_0^g(n)-\xi_{k,i}|\leq \frac{\delta(k)}{2}$. Hence $\min_{\|n\|\leq \|g\|k}|Q_0^g(n)-\xi|\leq \delta(k)$ as claimed.

\end{proof}

% ----------------------------------------------------------------
\bibliographystyle{alpha}
\bibliography{DKbibliog}
% ----------------------------------------------------------------
%GATHER{C:/localtexmf/Bib/Mybib.bib}   % For Gather Purpose Only

\end{document}